\documentclass{rose_v01}
\RequirePackage[latin1]{inputenc} \RequirePackage[T1]{fontenc}

\def\P{\mathbb P}
\def\R{\mathbb R}
\def\E{\mathbb E}

\def\Q{\mathbb Q}

\def\E{\mathbb E}
\def\N{\mathbb N}
\def\cal{\mathcal}
\title{L$^{p}$-solution of reflected generalized BSDEs with non-Lipschitz coefficients}
\headlinetitle{L$^{p}$-solution of BDSE and non-Lipschitz coefficients}

\authorone{Auguste Aman}
\addressone{UFR Mathématique Informatique\\Université de Cocody, 22 BP 582}
\countryone{Abidjan (Côte d'Ivoire)}
\emailone{augusteaman5@yahoo.fr}
\thanksone{}

\firstpage{12}                                      %
\doi{002}                                           %
\communicated{Name of Editor}                       %
\received{12 April, 2007}                           %
\revised{13 July, 2007}                             %

\abstract{In this paper, we continue in solving reflected generalized backward
stochastic differential equations (RGBSDE for short) and fixed terminal time with use some new technical aspects of
the stochastic calculus related to the reflected generalized BSDE.
Here, existence and uniqueness of solution is proved under the non-Lipschitz condition on the coefficients.  }



\acknowledgments{%
I would like to thank X, Y and Z for their valuable comments on earlier
drafts of this paper.}

\begin{document}
\section{Introduction}
The study of nonlinear backward stochastic differential equations (BSDEs, in short) was initiated by Pardoux and Peng \cite{PP1}. Mainly motivated by financial problems (see e.g. the survey article by El Karoui et al.\,\cite{Kal}), stochastic control and stochastic games (see the works by Hamad\`{e}ne and Lepeltier \cite{HL} and references therein ), the theory of BSDEs was developed at high speed during the 1990. These equations also provide
probabilistic interpretation for solutions to both elliptic and parabolic nonlinear partial differential equations (see Pardoux and Peng \cite{PP2},
Peng \cite{PE}). Indeed, coupled with a forward SDE, such BSDE's give an extension of the celebrate Feynman-Kac formula to nonlinear case.

In  order to provide a probabilistic representation for solution of parabolic or elliptic semi-linear PDEs with Neumann boundary condition, Pardoux and Zhang \cite{PZ} introduced the so-called generalized BSDEs. This equation involves the integral with respect to an increasing process.

El-Karoui et al. \cite{Kal1} have introduced the notion of reflected BSDEs (RBSDEs, in short). Actually, it is a BSDE, but one of the components of the solution is forced to stay above a given barrier. Since then, many others results on the RBSDEs have been established (see \cite{H,HML} and references therein) . In El-Karoui et al. \cite{Kal1}, the RBSDEs also provided a probabilistic formula for the viscosity solution of an obstacle problem for a parabolic PDEs.

Following this way, Ren et al \cite{Ral} have introduced the notion of reflected generalized BSDEs (RGBSDE, in short). They connected it to the obstacle problem for PDEs with Neumann boundary condition. More precisely, let consider the following RGBSDE: for $0\leq t\leq T$,
\begin{eqnarray}
&&(i)\, Y_{t}=\xi+\int_{t}^{T}f(s,Y_{s},Z_{s})ds+\int_{t}^{T}g(s,Y_{s})dG_{s}
-\int_{t}^{T}Z_{s}dW_{s}+K_{T}-K_{t}\nonumber\\
&&(ii)\, Y_{t}\geq S_{t}\label{aa}\\
&&(iii)\,K\, \mbox{is a non-decreasing process such that}\, K_0=0 \,\mbox{and}\,\int_{0}^{T}(Y_{t}-S_{t})dK_{t}=0.\nonumber
\end{eqnarray}

They proved under suitable conditions on the data the existence and uniqueness of the solution $(Y,Z,K)$. The increasing process $K$ is introduced to pushes the component $Y$ upwards so that it may remain above the obstacle process $S$. In particular, condition $(iii)$ means that the push is minimal and is done only when the constraint is saturated i.e. $Y_t< S_t$. In practice (finance market for example), the process $K$ can be regarded as the subsidy injected by a government in the market to allow  the price process $Y$ of a commodity  (coffee, by example) to remain above a threshold price process $S$.

In the Markovian framework, the RGBSDE $(\ref{aa})$ is combined with the following reflected forward SDE: for every $
(t,x)\in \lbrack 0,T]\times \overline{\Theta }$ and $s\in \lbrack t,T]$
\begin{equation*}
\left\{
\begin{array}{lll}
X_{s}^{t,x} & = & x+\displaystyle\int_{t}^{s\vee t}b(X_{r}^{t,x})dr+%
\displaystyle\int_{t}^{s\vee t}\sigma (X_{r}^{t,x})dW_{r}+\displaystyle%
\int_{t}^{s\vee t}\nabla \psi (X_{r}^{t,x})dG_{r}^{t,x},~s\geq 0 \\
X_{s}^{t,x} & \in & \overline{\Theta }\;\mbox{ and }\;G_{s}^{t,x}=%
\displaystyle\int_{t}^{s\vee t}1_{\left\{ X_{r}^{x}\in \partial \Theta
\right\} }dG_{r}^{t,x},%
\end{array}
\right.  \label{bb}
\end{equation*}

where $G_{.}^{t,x}$ is an increasing process and $\psi \in \mathcal{C}_{b}^{2}(%
\mbox{I\hspace{-.15em}R}^{d})$ characterize $\Theta $ and $\partial
\Theta $ as follows:
\begin{equation*}
\Theta =\{x\in \mathbb{R}^{d}:\;\;\psi (x)>0\}\quad \quad \mbox{and}\qquad
\qquad \partial \Theta =\{x\in \mathbb{R}^{d}:\;\;\psi (x)=0\}.
\end{equation*}

Assuming the data in the form $\xi =l(X_{T}^{t,x})$, $%
S_{s}=h(s,X_{s}^{t,x})$, $f(s,y,z)=f(s,X_{s}^{t,x},y,z),\mbox{ }$ and $%
g(s,y)=g(s,X_{s}^{t,x},y),$ the RGBSDE $(\ref{aa})$ becomes: for any fixed $t\in \lbrack
0,T]$
\begin{equation}
\label{C1}\left\{
\begin{array}{l}
\left( i\right) \mbox{ }Y_{s}^{t,x}=l(X_{T}^{t,x})+\displaystyle%
\int_{s}^{T}f(r,X_{r}^{t,x},Y_{r}^{t,x},Z_{r}^{t,x})dr+\displaystyle%
\int_{s}^{T}g(r,X_{r}^{t,x},Y_{r}^{t,x})dG_{r}^{t,x} \\
\\
~~\ \ \ \ ~~~~~\ ~~\ \ \ ~\ \ \ \ \ -\displaystyle%
\int_{s}^{T}Z_{r}^{t,x}dW_{r}+K_{T}^{t,x}-K_{s}^{t,x},\mbox{
}s\in \lbrack t,T] \\
\\
\left( ii\right) \mbox{ }Y_{s}^{t,x}\geq h(s,X_{s}^{t,x}),\mbox{
} \; a.s., \forall\ s\in \lbrack t,T] \\
\\
\left( iii\right) \mbox{ }K^{t,x}\mbox{ is a non-decreasing
process such
that }K_{0}^{t,x}=0\mbox{ and }\displaystyle%
\int_{t}^{T}(Y_{s}^{t,x}-h(s,X_{s}^{t,x}))dK_{s}^{t,x}=0, \ a.s.,
\end{array}%
\right.
\end{equation}

and gives a probabilistic
interpretation of the following type of obstacle problem for a partial
differential equation with nonlinear Neumann boundary condition:
\begin{equation*}
\left\{
\begin{array}{l}
\min \left\{ u\left( t,\,x\right) -h\left( t,\,x\right) ,\right. \\
\\
\left. -\frac{\partial u}{\partial t}\left( t,\,x\right) -(L
u)\left( t,\,x\right) -f(s,\,x,\,u\left( t,\,x\right) ,\,\left( \nabla
u\left( t,\,x\right) \right) ^{\ast }\sigma \left( t,\,x\right) )\right\} =0,
\\
\\
\left( t,\,x\right) \in \left[ 0,\,T\right] \times \Theta \\
\\
\frac{\partial u}{\partial n}\left( t,\,x\right) +g\left( t,\,x,\,u\left( t,\,x\right) \right) =0,
\mbox{ }\left( t,\,x\right) \in \left[ 0,\,T\right] \times \partial \Theta
\\
\\
u\left( T,\,x\right) =l\left( x\right) ,\mbox{ }x\in \overline{\Theta },
\end{array}%
\right.  \label{DD}
\end{equation*}
where $L$ is the infinitesimal generator corresponding to the diffusion process $X^{x}$ and $\frac{\partial}{\partial n}(.)=\langle\nabla\psi,\nabla (.)\rangle$.

Apart the work of El Karoui et al. \cite{Kal} and Briand et al. \cite{BDPS} in the case of standard BSDEs, there has been relatively few papers which deal
with the problem of existence and/or uniqueness of the solution for BSDEs and RBSDEs in the case when the coefficients are not square integrable.  This limits the scope for several applications (finance, stochastic control,  stochastic games, PDEs, etc,$\cdot\cdot$). To correct this shortcoming, Hamad\`{e}ne and Popier \cite{HP} show that if $\xi, \sup_{0\leq t\leq T} (S^+_t )$ and $\int_0^T |f(t, 0, 0)|dt$ belong to $L^p$ for some $p\in]1, 2[$, then the RBSDEs with one reflecting barrier associated with $(f,\ g=0,\ \xi,\ S)$ has a unique solution. They prove existence and uniqueness of the solution in using penalization and Snell envelope of processes methods. In a previous works, Aman \cite{Aman} give the similar result for a class of RGBSDEs $(\ref{aa})$ with Lipschitz condition on the coefficients by used the $L^{\infty}$-approximation. In this paper, we extend the previous result, assuming that  in this case coefficients are non-Lipschitz. The rest of the paper is organized as follows. The next section contains all the notations, assumptions and a priori estimates. Section 3 is devoted to existence and uniqueness result in $L^{p}$,\, $p\in(1, 2)$ when the coefficients are non-Lipschitz.

\section{ Preliminaries}
\setcounter{theorem}{0} \setcounter{equation}{0}
\subsection{Assumptions and basic notations}
First of all, $W=\{W_{t}\}_{t>0}$ is a standard Brownian motion with values in $\R^d$ defined on
some complete probability space $\left(\Omega ,\mathcal{F},\P\right)$. $\{\mathcal{ F}_{t}\}_{t\geq 0}$ is the augmented natural filtration
of $W$ which satisfies the usual conditions. In this paper, we will always use this filtration.
In most of this work, the stochastic processes will be defined for $t\in[0,T]$, where
$T$ is a positive real number, and will take their values in $\R$.

For any real $p>0$,
let us define the following spaces:\newline
$\mathcal{S}^{p}(\R)$ denotes set of $\R$-valued, adapted
c\`{a}dl\`{a}g processes $
 \{X_{t}\}_{t\in \lbrack 0,T]}$ such that
\begin{equation*}
\|X\|_{\mathcal{S}^{p}}=\E\left( \sup\limits_{0\leq t\leq
T}|X_{t}|^{p}\right) ^{1\wedge \frac{1}{p}}<+\infty,
\end{equation*}
and $\mathcal{M}^{p}(\R^{d})$ is the set of  predictable processes
$\{X_{t}\}_{t\in \lbrack 0,T]}$  such that
\begin{eqnarray*}
\|X\|_{\mathcal{M}^{p}}=\E \left[ \left(
\int_{0}^{T}|X_{t}|^{2}dt\right) ^{\frac{p}{2}}\right] ^{1\wedge \frac{1}{p}%
}<+\infty .
\end{eqnarray*}
If $p\geq 1$, then $\|X\|_{\mathcal{S}^{p}}$ (resp
$\|X\|_{\mathcal{M}^{p}}$) is
a norm on $\mathcal{S}^{p}(\R)$ (resp. $\mathcal{M}
^{p}(\R^{d}))$ and these spaces are Banach spaces. But if $p\in
\left( 0,1\right) ,$ $\left( X,X^{\prime }\right) \longmapsto
\left\| X-X^{\prime }\right\| _{\mathcal{S}^{p}}$ (resp
$\left\|
X-X^{\prime }\right\| _{\mathcal{M}^{p}}$) defines a distance on $\mathcal{S}%
^{p}(\R),$ (resp. $\mathcal{M}^{p}(%
\R^{d})$) and under this metric, $\mathcal{S}^{p}(%
\R)$ (resp. $\mathcal{M}^{p}(\R^{d}))$ is complete.

Now let us give the following assumptions:

\begin{description}
\item  $\left( \text{\bf A1}\right) $ $\left( G_{t}\right) _{t\geq
0}$ is a continuous real valued increasing
$\mathcal{F}_{t}$-progressively measurable process with bounded
variation on $[0,T]$.

\item  $\left( \text{\bf A2}\right)$ Two functions $f:\Omega\times[0,T]\times \mbox{I\hspace{-.15em}R}\times\mbox{I\hspace{-.15em}R}^{d}\rightarrow \R$ and
$g:\Omega \times[0,T]\times \mbox{I\hspace{-.15em}R}\rightarrow\R$ for some constants $\beta <0,\ \lambda >0,\ \mu\in\R$ and for all $t\in[0,T],\, y, y'\in\R,\; z, z'\in\R^d$:

\begin{description}
\item $(i)\; y\longmapsto (f(t,y,z),g(t,y))$ is continuous for all $z,\ (t,\omega)$ a.e.,
\item $(ii)\; f(.,y,z)$ and $g(.,y)$ are progressively measurable,
\item $(iii)\; |f(t,y,z)-f(t,y,z^{\prime })|\leq \lambda |z-z^{\prime }|$,
\item $(iv)\; \left( y-y^{\prime}\right) \left( f(t,y,z)-f(t,y^{\prime },z)\right) \leq \mu|y-y^{\prime }|^{2}$,
\item $(v)\; |f(t,y,z)|\leq |f(t,0,0)|+M(|y|+|z|)$
\item $(vi)\;\left( y-y^{\prime}\right) \left( g(t,y)-g(t,y^{\prime })\right) \leq \beta |y-y^{\prime }|^{2}$,
\item $(vii)\; | g(t,y)|\leq  |g(t,0)|+M|y|$,
\item $(viii)\; \E \left[\left(\int_{0}^{T}|f(s,0,0)|ds\right)^{p}+ \left( \int_{0}^{T}|g(s,0)|dG_{s}\right)^{p}\right]<\infty$.
\end{description}

\item  $\left({\bf A3}\right)$ For any $r>0$, we define the process $\pi
 _{r}$ in $L^{p}\left(\left[ 0,T\right]\times \Omega
 ,m\otimes \P\right)$ by
 \begin{eqnarray*}
\pi _{r}(t)=\sup_{|y|\leq r}|f(t,y,0)-f(t,0,0)|.
\end{eqnarray*}
\item  $\left( \text{\bf A4}\right)\xi $ is a
$\mathcal{F}_{T}$-measurable variable such that $\E(|\xi
|^{p})<+\infty$.

\item  $\left( \text{\bf A5}\right)$ There exists a barrier $\left( S_{t}\right) _{t\geq
0}$ which is a continuous, progressively
measurable, real-valued process satisfying:
\begin{description}
\item $(i)\; \E\left( \sup_{0\leq t\leq T}(S_{t}^{+})^{p}\right) <+\infty$,
\item $(ii)\; S_{T}\leq \xi\;\;\;\;\; \P$- a.s.
\end{description}
\end{description}

Before of all, let us recall what we mean by a $L^{p}$-solution of
RGBSDEs.
\begin{definition}
A $L^{p}$-solution of RGBSDE associated to the
data $(\xi,f,g,S)$  is a triplet  $(Y_{t},Z_{t},K_{t})_{0\leq t \leq
T}$  of progressively measurable processes taking values in $\R
\times \R^{d} \times\R$ and satisfying:
\begin{description}
\item $\left(i\right)$  $Y$  is a continuous process, \item
$\left(ii\right)$\begin{eqnarray}
Y_{t}=\xi
+\int_{t}^{T}f(s,Y_{s},Z_{s})ds+\int_{t}^{T}g(s,Y_{s})dG_{s}
-\int_{t}^{T}Z_{s}dW_{s}+K_{T}-K_{t},\label{RGBDSEinitial}
\end{eqnarray}
 \item $\left(iii\right)$\
$Y_{t}\geq S_{t}$ \;\; a.s., \item $\left(iv\right)$ $\E\left(\sup_{0\leq
t\leq T}|Y_{t}|^{p}+\left(\int_{0}^{T}|Z_{s}|^{2}ds\right)^
{p/2}\right)<+\infty$, \item$\left(v\right)$\ $K$ is a non-decreasing
process such that $K_{0}=0$ and
$\int_{0}^{T}(Y_{s}-S_{s})dK_{s}=0$,\;\; a.s.
\end{description}
\end{definition}

\subsection{A priori estimates}
In this paragraph, we state some estimates for solution of RGBSDE associated to $\left( \xi ,f,g,S\right) $ in $L^{p}$ when $p>1$ like in \cite{Aman}. But the difficulty here comes from the facts the function $f$ is not supposed to be Lipschitz continuous. Let us give the notation $\widehat{x}=|x|^{-1}x\mathbf{1}_{\{x\neq
0\}}$ introduced in \cite{BDPS} that will play an important role in the sequel.

\begin{lemma}
Assume that $(Y, Z)\in\mathcal{S}^{p}(\R)\times\mathcal{M}^{p}(\R^{d})$ is a solution of the following
BSDE:
\begin{eqnarray}
Y_{t}=\xi
+\int_{t}^{T}\tilde{f}(s,Y_{s},Z_{s})ds+\int_{t}^{T}\tilde{g}(s,Y_{s})dG_{s}
-\int_{t}^{T}Z_{s}dW_{s}+A_{T}-A_{t},\;\; 0\leq t\leq T,
\end{eqnarray}
where

\begin{description}
\item $(i)\; \tilde{f}$ and $\tilde{g}$ are functions which satisfy assumptions\,$({\bf A2})$,
\item $(ii)\, \P$ a.s., the process $(A_t)_{0\leq t\leq T}$ is of bounded variation type.
\end{description}

Then for any $0 \leq t \leq T$ we have:
\begin{eqnarray*}
&&|Y_{t}|^{p}+c(p)\int_{t}^{T}|Y_{s}|^{p-2}\mathbf{1}_{\{Y_{s}\neq 0\}}
|Z_s|^{2}ds\\
&\leq &|\xi |+p\int_{t}^{T}|Y_{s}|^{p-1}\widehat{Y}_{s}\ \tilde{f}\left(
s,Y_{s},Z_{s}\right) ds+p\int_{t}^{T}|Y_{s}|^{p-1}
\widehat{Y}_{s}\ \tilde{g}\left( s,Y_{s}\right)  dG_{s} \\
&&+p\int_{t}^{T}|Y_{s}|^{p-1} \widehat{Y}_{s}\ dA_{s}
-p\int_{t}^{T}|Y_{s}|^{p-1} \widehat{Y}_{s}\ Z_{s}dW_{s} .
\end{eqnarray*}
with $c(p)=p\left[ (p-1)\wedge 1\right] /2.$
\end{lemma}
We now show how to control the process $Z$ in terms of the data and the process $Y$.
\begin{lemma}\label{L2}
Let assume $\left({\bf A1}\right)$-$\left(
{\bf A4}\right) $ hold and let $\left( Y,Z,K\right)$ be the solution of RGBSDE associated to $\left( \xi
,f,g,S\right) .$ If $Y\in \mathcal{S}^{p}$ then $Z$ belong to $\mathcal{M}%
^{p}$ and there exists a real constant $C_{p,\lambda}$ depending only on $p$ and $\lambda$
such that,
\begin{eqnarray*}
\E\left[ \left( \int_{0}^{T}|Z_{r}|^{2}dr\right)
^{p/2}\right]
  &\leq &C_{p,\lambda}%
\E\left\{ \sup_{0\leq t\leq T}|Y_{t}|^{p}+\left(
\int_{0}^{T}f^{0}_{r}dr\right) ^{p}\right. \\
&&+\left. \left( \int_{0}^{T}g^{0}_{r}dG_{r}\right)
^{p}+\sup_{0\leq t\leq T}|S^{+}_{t}|^{p}\right\},
\end{eqnarray*}
where $f^{0}_{r}=|f(r,0,0)|$ and $g^{0}_{r}=|g(r,0)|$.
\end{lemma}
\begin{proof}
For each integer $n\geq 1$ let introduce
\begin{eqnarray*}
\tau_{n}=\inf \left \{t\in [0,T],\int_{0}^{t}|Z_{r}|^2dr \geq n \right \}
 \wedge T.
\end{eqnarray*}
The sequence $(\tau_n)_{n\geq 0}$ is of stationary type since the process $Z$ belongs to $\mathcal{M}^{p}$ and then $\int^{T}_{0} |Z_s|^{2}ds <\infty,\, \P$- a.s..
Next, for any $\alpha>0$, using It\^o's formula and assumption $(\bf A2)$, we get
\begin{eqnarray*}
&&|Y_{0}|^{2}+\int_{0}^{\tau_{n}}e^{\alpha r}|Z_{r}|^{2}dr+|\beta|\int_{0}^{\tau_{n}} e^{\alpha r}|Y_{r}|^{2}dG_r\nonumber\\
&\leq&e^{\alpha \tau_{n}}|Y_{\tau_{n}}|^{2}+2\sup_{0\leq t\leq T}e^{\alpha t}|Y_{t}|\times\left[\int_{0}^{\tau_{n}} (f^{0}_{r}dr+g^{0}_{r}dG_r)\right]
+(2\lambda+\varepsilon^{-1}\lambda-\alpha)\int_{0}^{\tau_{n}} e^{\alpha r}|Y_{r}|^{2}dr\\
&&+\varepsilon\int_{0}^{\tau_{n}} e^{\alpha r}|Z_r|^{2}dr + \frac{1}{\varepsilon}\sup_{0\leq t\leq \tau_{n}}e^{2\alpha t}|Y_{t}|^{2}+\varepsilon|K_{\tau_{n}}|^{2} - 2\int_{0}^{\tau_{n}}e^{\alpha r}Y_{r}\ Z_{r}dW_{r},
\end{eqnarray*}
in virtue of the standard inequality $2ab\leq \frac{1}{\varepsilon}a^{2}+\varepsilon b^{2}$ for any $\varepsilon>0$ and since $\beta<0$.

But
\begin{eqnarray}
|K_{\tau_{n}}|^{2}&\leq& C_{\lambda}\left\{|Y^{2}_{0}|+|Y^{2}_{\tau_{n}}|+\left(\int_{0}^{\tau_{n}}f^{0}_{r} dr
\right)^{2}+\int_{0}^{\tau_{n}}|Y_{r}|^{2}dr+\int_{0}^{\tau_{n}}|Y_{r}|^{2}dG_r\right.\nonumber\\
&&\left.+\left(\int_{0}^{\tau_{n}}g^{0}_{r} dG_{r}\right)^{2}+\int_{0}^{\tau_{n}}|Z_{r}|^{2}dr+\left|\int_{0}^{\tau_{n}}Z_rdW_r\right|
\right\}\label{estK}
\end{eqnarray}
so that we have:
\begin{eqnarray*}
&&(1-\varepsilon C_{\lambda})|Y_{0}|^{2}+(1-\varepsilon-\varepsilon C_{\lambda})\int_{0}^{\tau_{n}}e^{\alpha r}|Z_{r}|^{2}dr\nonumber\\
&&\leq (\varepsilon C_{\lambda}+e^{\alpha \tau_{n}})|Y_{\tau_{n}}|^{2}+(1+\varepsilon C_{\lambda})\left[\left(\int_{0}^{\tau_{n}} f^{0}_{r}dr\right)^{2}+\left(\int_{0}^{\tau_{n}}g^{0}_{r}dG_r\right)^{2}\right]\\
&&+(2\lambda+\varepsilon^{-1}\lambda-\alpha)\int_{0}^{\tau_{n}} e^{\alpha r}|Y_{r}|^{2}dr+(1+\frac{1}{\varepsilon})\sup_{0\leq t\leq \tau_{n}}e^{2\alpha t}|Y_{t}|^{2}\\
&&+\varepsilon C_{\lambda}\left|\int_{0}^{\tau_{n}}Z_{r}dW_{r}\right|+ 2\left|\int_{0}^{\tau_{n}}e^{\alpha r}Y_{r}\ Z_{r}dW_{r}\right|.
\end{eqnarray*}
Choosing now $\varepsilon$ small enough and $\alpha$ such that $2\lambda+ \varepsilon^{-1}\lambda-\alpha< 0$, we obtain:
\begin{eqnarray*}
\left(\int_{0}^{\tau_{n}}|Z_{r}|^{2}dr\right)^{p/2} &\leq&
C_{p,\lambda}\left\{\sup_{0\leq t\leq\tau_{n}}Y_{t}^{p}+\left(\int_{0}^{\tau_{n}}f^{0}_{r} dr\right)^{p}\right.\\
&&\left.+\left(\int_{0}^{\tau_{n}}g^{0}_{r} dG_{r}\right)^{p}+\left|\int_{0}^{\tau_{n}}e^{\alpha r}Y_{r}\ Z_{r}dW_{r}\right|^{p/2}\right\}. \label{a2}
\end{eqnarray*}
Next thanks to BDG's inequality it follows:
\begin{eqnarray*}
\E\left(\left|\int_{0}^{\tau_{n}}e^{\alpha r}Y_{r}Z_{r}dW_{r}
\right|^{p/2}\right)&\leq& d_{p}\E\left[ \left(\int_{0}^{\tau_{n}}
|Y_{r}|^{2}|Z_{r}|^{2}dr\right)^{p/4}\right]\nonumber \\
&\leq& \bar{C}_{p}\E\left[ \sup_{0\leq t\leq \tau_{n}}|Y_{t}|^{p/2}\left(\int_{0}^{\tau_{n}}
|Z_{r}|^{2}dr\right)^{p/4}\right]\nonumber\\
&\leq& \frac{\bar{C}_{p}^{2}}{\eta}\E\left(\sup_{0\leq t\leq \tau_{n}}|Y_{t}|^{p}\right)+
\eta\E\left(\int_{0}^{\tau_{n}}|Z_{r}|^{2}dr\right)^{p/2}.
\label{a3}
\end{eqnarray*}
Finally plugging the last inequality in the previous one, choosing $\eta$ small enough and
finally using Fatou's lemma we obtain the desired result.
\end{proof}

We will now establish an estimate for the processes $Y$ and $Z$. The difficulty comes from the fact that the function $y\mapsto |y|^p$ is not $\mathcal{C}^{2}$ since we work with $p\in (1,2)$. Actually we have:
\begin{lemma}
Assume $\left({\bf A1}\right)$-$\left({\bf A4}\right)$. Let
$\left(Y,Z,K)\right)$ be a solution of the RGBDSE
associated to the data $(\xi,f,g,S)$ where $Y$ belong to
$\mathcal{S}^{p}$. Then there exists a constant $C_{p,\lambda}$ depending
only on $p$ and $\lambda$ such that
\begin{eqnarray*}
\E\left\{\sup_{0\leq t \leq T}|Y_{t}|^{p}+
\left(\int_{0}^{T}|Z_{s}|^{2}ds\right)^{p/2}
\right\}&\leq& C_{p,\lambda}\E\left\{|\xi|^{p}+\left(\int_{0}^{T}f^{0}_{s}ds\right)^{p}\right.\\
&&+\left.\left(\int_{0}^{T}g^{0}_{s}dG_{s}\right)^{p}+ \sup_{0\leq
t \leq T}(S_{t}^{+})^{p} \right\}.
\end{eqnarray*}
\end{lemma}
\begin{proof}
For any $\alpha>0$, it from Lemma 2.2, together with assumption $({\bf A2})$ that
\begin{eqnarray*}
&&e^{p\alpha t}|Y_{t}|^{p}+
c(p)\int_{t}^{u}e^{p\alpha s}|Y_s|^{p-2}{\bf 1}_{\{Y_s\neq\ 0\}}|Z_{s}|^{2}ds\\
&\leq& e^{p\alpha u}|Y_u|^{p}+p(\lambda-\alpha)\int_{u}^{T}e^{p\alpha s}|Y_s|^{p}ds+p\int_{t}^{u}e^{p\alpha s}|Y_s|^{p-1}f^{0}_sds\\
&&+p\int_{t}^{u}e^{p\alpha s}|Y_s|^{p-1}g_{s}^{0}dG_s+p\lambda\int_{u}^{T}e^{p\alpha s}|Y_s|^{p-1}|Z_s|ds\\
&&+p\int_{t}^{u}e^{p\alpha s}|Y_s|^{p-1}\widehat{Y}_sdK_s-p\int_{t}^{u}e^{p\alpha s}|Y_s|^{p-1}\widehat{Y}_sZ_sdW_s.
\end{eqnarray*}
We have by Young's inequality
\begin{eqnarray*}
p\lambda|Y_s|^{p-1}|Z_s|\leq \frac{p\lambda^{2}}{p-1}|Y_s|^{p}+\frac{c(p)}{2}|Y_s|^{p-2}{\bf 1}_{\{Y_s\neq 0\}}|Z_s|^{2},
\end{eqnarray*}
and
\begin{eqnarray*}
p\int_{t}^{u}e^{p\alpha s}|Y_s|^{p-1}(f^{0}_sds+g^{0}_sdG_s)&\leq &(p-1)\gamma^{\frac{p}{p-1}}\sup_{0\leq s\leq u}|Y_s|^{p}\\
&&+\gamma^{-p}\left[\left(\int_{t}^{u}e^{p\alpha s}f^{0}_sds\right)^{p}+\left(\int_{t}^{u}e^{p\alpha s}g^{0}_sdG_s\right)^{p}\right]
\end{eqnarray*}
for any $\gamma>0$. Then plug the two last inequalities in the previous one, we obtain:
\begin{eqnarray*}
&&e^{p\alpha t}|Y_{t}|^{p}+
\frac{c(p)}{2}\int_{t}^{u}e^{p\alpha s}|Y_s|^{p-2}{\bf 1}_{\{Y_s\neq\ 0\}}|Z_{s}|^{2}ds\\
&\leq& e^{p\alpha u}|Y_u|^{p}+(p-1)\gamma^{\frac{p}{p-1}}\sup_{0\leq s\leq u}|Y_s|^{p}\\
&&+\gamma^{-p}\left[\left(\int_{t}^{u}e^{p\alpha s}f^{0}_sds\right)^{p}+\left(\int_{t}^{u}e^{p\alpha s}g^{0}_sdG_s\right)^{p}\right]\\
&&+p\left(\lambda+\frac{\lambda^{2}}{p-1}-\alpha\right)\int_{t}^{u}e^{p\alpha s}|Y_s|^{p}ds\\
&&+p\int_{t}^{u}e^{p\alpha s}|Y_s|^{p-1}\widehat{Y}_sdK_s-p\int_{t}^{u}e^{p\alpha s}|Y_s|^{p-1}\widehat{Y}_sZ_sdW_s.
\end{eqnarray*}
Next, the hypothesis related to increments of $K$ and $Y-S$ implies that
\begin{eqnarray*}
\int_{t}^{u}e^{p\alpha s}|Y_s|^{p-1}\widehat{Y}_sdK_s&\leq& \int_{t}^{u}e^{p\alpha s}|S_s|^{p-1}\widehat{S}_sdK_s\\
&\leq &\int_{t}^{u}e^{p\alpha s}(S_s^{+})^{p-1}dK_s\\
&\leq&\frac{p-1}{p}\frac{1}{\varepsilon^{\frac{p}{p-1}}}\left(\sup_{0\leq t\leq u}|S^{+}_t|^{p}\right)+\frac{1}{p}\varepsilon^{p}\left(\int_{t}^{u}e^{p\alpha s}dK_s\right)
\end{eqnarray*}
for any $\varepsilon>0$, so that choosing $\alpha$ such that
$\begin{array}{l}
\lambda+\frac{\lambda^{2}}{p-1}\leq\alpha
\end{array}
$
and put $u=T$, we get:
\begin{eqnarray}
&&\E\left(e^{p\alpha t}|Y_{t}|^{p}\right)+
\frac{c(p)}{2}\E\left(\int_{t}^{T}e^{p\alpha s}|Y_s|^{p-2}{\bf 1}_{\{Y_s\neq\ 0\}}|Z_{s}|^{2}ds\right)\nonumber\\
&\leq& \E(e^{p\alpha T}|\xi|^{p})+(p-1)\gamma^{\frac{p}{p-1}}\E\left(\sup_{0\leq s\leq T}|Y_s|^{p}\right)\nonumber\\
&&+\gamma^{-p}\E\left[\left(\int_{t}^{T}e^{p\alpha s}f^{0}_sds\right)^{p}+\left(\int_{t}^{T}e^{p\alpha s}g^{0}_sdG_s\right)^{p}\right]\label{estY}\\
&&+(p-1)\frac{1}{\varepsilon^{\frac{p}{p-1}}}\E\left(\sup_{0\leq t\leq T}|S^{+}_t|^{p}\right)+\frac{1}{p}\varepsilon^{p}\E\left(\int_{t}^{T}e^{p\alpha s}dK_s\right).\nonumber
\end{eqnarray}
On the other hand the predictable dual projection, Jensen's conditional inequality and together with Lemma \ref{L2} provide
\begin{eqnarray}
\E[(K_T-K_t)^{p}]\leq C_{\lambda,p}\E\left[\sup_{0\leq s\leq T}|Y_s|^{p}+\left(\int_{t}^{T}f_s^{0}ds\right)^{p}+\left(\int_{t}^{T}g_s^{0}dG_s\right)^{p}\right],
\label{estY1}
\end{eqnarray}
where $C_{\lambda,p}$ is a constant which depend on $p,\ \lambda$ and possibly $T$ which may change from line to another.

Coming back to inequality $(\ref{estY})$ and using BDG inequality we have
\begin{eqnarray*}
\E\sup_{0\leq t\leq T}e^{p\alpha t}|Y_{t}|^{p}&\leq& \E(e^{p\alpha T}|\xi|^{p})+(p-1)\frac{1}{\varepsilon^{\frac{p}{p-1}}}\E\left(\sup_{0\leq t\leq T}|S^{+}_t|^{p}\right)\\
&&+\{C_{\lambda,p}(\gamma^{\frac{p}{p-1}}+\varepsilon^{p})+p\eta\}\E\left(\sup_{0\leq t\leq T}|Y_t|^{p}\right)\\
&&+C_{\lambda,p}(\frac{1}{\gamma^{p}}+\varepsilon^{p})\E\left[\left(\int_{0}^{T}e^{p\alpha s}f^{0}_sds\right)^{p}+\left(\int_{0}^{T}e^{p\alpha s}g^{0}_sdG_s\right)^{p}\right]\\
&&+\frac{p}{\eta}\E\left(\int_{0}^{T}e^{p\alpha s}|Y_s|^{p-2}\widehat{Y}_s{\bf 1}_{\{Y_s\neq 0\}}|Z_s|^{2}ds\right)\\
&\leq&\left(1+\frac{2p}{c(p)\eta}\right)\E(e^{p\alpha T}|\xi|^{p})+\left(1+\frac{2p}{c(p)\eta}\right)(p-1)\frac{1}{\varepsilon^{\frac{p}{p-1}}}\E\left(\sup_{0\leq t\leq T}|S^{+}_t|^{p}\right)\\
&&+\left(1+\frac{2p}{c(p)\eta}\right)C_{\lambda,p}(\frac{1}{\gamma^{p}}+\varepsilon^{p})\E\left[\left(\int_{0}^{T}e^{p\alpha s}f^{0}_sds\right)^{p}+\left(\int_{0}^{T}e^{p\alpha s}g^{0}_sdG_s\right)^{p}\right]\\
&&+\left\{C_{\lambda,p}\left(1+\frac{2p}{c(p)\eta}\right)(\gamma^{\frac{p}{p-1}}+\varepsilon^{p})+p\eta\right\}\E\left(\sup_{0\leq t\leq T}|Y_t|^{p}\right)\\
\end{eqnarray*}
Finally it is enough to chose $\eta=\frac{1}{2p}$ and $\gamma,\,\varepsilon$ small enough to obtain the desired result.
\end{proof}
\begin{lemma}
Assume that $(f, g,\xi, S)$ and $(f',g', \xi',S')$ are two quadruplets satisfying assumptions $({\bf A1})$-$({\bf A4})$.
Suppose that $(Y, Z,K)$ is a solution of RGBSDE $(f, g,\xi, S)$ and $(Y', Z',K')$ is a solution of RGBSDE $(f',g',\xi',S')$. Let us set:
\begin{eqnarray*}
\Delta f = f-f',\; \Delta\xi = \xi-\xi',\; \Delta S = S-S'\\
\Delta Y = Y-Y',\; \Delta Z = Z-Z,\, \Delta K = K-K'
\end{eqnarray*}
and assume that $\Delta S\in L^{p}(dt\times\P)$. Then there exists a constant C such that
\begin{eqnarray*}
\E \left(\sup_{t\in[0,T]}|\Delta Y_{t}|^{p}\right)&\leq& C\E\left[|\Delta\xi|^{p}+\left(\int^{T}_{0}|\Delta f(s, Y_s, Z_s)|ds\right)^{p}\right]\\
&&+\left(\int^{T}_{0}|\Delta g(s, Y_s)|dG_s\right)^{p}+C(\Psi(T))^{1/p}\E\left[\sup_{t\in[0,T]}|\Delta S_{t}|^{p}\right]^{\frac{p-1}{p}},
\end{eqnarray*}
with
\begin{eqnarray*}
\Psi(T)&=& \E\left[|\xi|^{p} +\left(\int^{T}_{0}f^{0}_sds\right)^{p}+\left(\int^{T}_{0}g^{0}_sdG_{s}\right)^{p}+\sup_{t\in[0,T]}(S^{+}_{t})^{p}\right.\\
&&\left.+|\xi'|^{p} +\left(\int^{T}_{0}f'^{0}_sds\right)^{p}+\left(\int^{T}_{0}g'^{0}_sdG_{s}\right)^{p}+\sup_{t\in[0,T]}(S'^{+}_{t})^p\right].
\end{eqnarray*}
\end{lemma}
\begin{proof}
Using Lemma 2.1 and $({\bf A2})$ we have for all $0\leq t\leq T$:
\begin{eqnarray}
&&|\Delta Y_t|^{p}+c(p)\int_{t}^{T}|\Delta Y_s|^{p-2}{\bf 1}_{\{\Delta Y_s\neq 0\}}|\Delta Z_{s}|^{2}ds\nonumber\\
&\leq& |\Delta\xi|^{p}+p\lambda\int_{t}^{T}|\Delta Y_{s}|^{p-1}\widehat{\Delta Y_s}|\Delta Z_s|ds\nonumber\\
&&+p\lambda\int_{t}^{T}|\Delta Y_{s}|^{p}ds+p\int_{t}^{T}|\Delta Y_{s}|^{p-1}\widehat{\Delta Y_s}|\Delta f(s,Y_s,Z_s)|ds\label{dif}\\
&&+p\beta\int_{t}^{T}|\Delta Y_{s}|^{p}dG_s+p\int_{t}^{T}|\Delta Y_{s}|^{p-1}\widehat{\Delta Y_s}|\Delta g(s,Y_s)|dG_s\nonumber\\
&&+p\int_{t}^{T}|\Delta Y_{s}|^{p-1}\widehat{\Delta Y_s}d(\Delta K_s)-p\int_{t}^{T}|\Delta Y_{s}|^{p-1}\widehat{\Delta Y_s}\Delta Z_sdW_s.\nonumber
\end{eqnarray}
Moreover
\begin{eqnarray*}
\int_{t}^{T}|\Delta Y_{s}|^{p-1}\widehat{\Delta Y_s}d (\Delta K_s)&\leq&\int_{t}^{T}|\Delta S_s|^{p-2}(\Delta S_s){\bf 1}_{\{\Delta S_s\neq 0\}}dK_s\\
&&-\int_{t}^{T}|\Delta S_s|^{p-2}(\Delta S_s){\bf 1}_{\{\Delta S_s\neq 0\}}dK'_s\\
&\leq&\int_{t}^{T}|\Delta S_{s}|^{p-1}d (\Delta K_s)
\end{eqnarray*}
Thus coming back to $(\ref{dif})$ and thanks to the Burkholder-Davis-Gundy and Young inequalities, we get with $t=0$
\begin{eqnarray}
&&\frac{c(p)}{2}\E\int_{0}^{T}|\Delta Y_s|^{p-2}{\bf 1}_{\{\Delta Y_s\neq 0\}}|\Delta Z_{s}|^{2}ds\nonumber\\
&\leq &\E|\Delta\xi|^{p}+(\frac{p\lambda^{2}}{p-1}+p\lambda)\E\int_{0}^{T}|\Delta Y_{s}|^{p}ds\nonumber\\
&&+p\E\int_{0}^{T}|\Delta Y_{s}|^{p-1}|\Delta f(s,Y_s,Z_s)|ds+p\E\int_{0}^{T}|\Delta Y_{s}|^{p-1}|\Delta g(s,Y_s)|dG_s\nonumber\\
&&+p\E\int_{0}^{T}|\Delta S_{s}|^{p-1}d(\Delta K_s)\label{est1}
\end{eqnarray}
and
\begin{eqnarray}
\E|\Delta Y_t|^{p}&\leq &\E|\Delta\xi|^{p}+(\frac{p\lambda^{2}}{p-1}+p\lambda)\E\int_{0}^{T}|\Delta Y_{s}|^{p}ds\nonumber\\
&&+p\E\int_{0}^{T}|\Delta Y_{s}|^{p-1}|\Delta f(s,Y_s,Z_s)|ds+p\E\int_{0}^{T}|\Delta Y_{s}|^{p-1}|\Delta g(s,Y_s)|dG_s\nonumber\\
&&+p\E\int_{0}^{T}|\Delta S_{s}|^{p-1}d(\Delta K_s),\label{est2}
\end{eqnarray}
since we recall again $\beta<0$.

We have by holder's inequality
\begin{eqnarray*}
\E\int_{0}^{T}|\Delta S_{s}|^{p-1}d(\Delta K_s)\leq\left(\E\sup_{0\leq t \leq T}|\Delta S_t|^{p}\right)^{\frac{p}{p-1}}(\Psi_T)^{1/p}
\end{eqnarray*}
and
\begin{eqnarray*}
&&p\E\int_{0}^{T}|\Delta Y_{s}|^{p-1}|\Delta f(s,Y_s,Z_s)|ds+p\E\int_{0}^{T}|\Delta Y_{s}|^{p-1}|\Delta g(s,Y_s)|dG_s\\
&\leq &\gamma\E\sup_{0\leq t\leq T}|\Delta Y_t|^{p}+\frac{1}{\gamma}\E\left[\left(\int_{0}^{T}|\Delta f(s,Y_s,Z_s)|ds\right)^{p}+\left(\int_{0}^{T}|\Delta g(s,Y_s)|dG_s\right)^{p}\right]
\end{eqnarray*}
for any $\gamma>0$. Finally, return again to $(\ref{dif})$ and use  again Burkholder-Davis-Gundy together with inequalities $(\ref{est1})$ and $(\ref{est2})$, it follows
after choosing $\gamma$ small enough:
\begin{eqnarray*}
\E\left(\sup_{0\leq t\leq T}|\Delta Y_t|^{p}\right)&\leq& CE\left[|\Delta\xi|^{p}+\left(\int_{0}^{T}|\Delta f(s,Y_s,Z_s)|ds\right)^{p}+\left(\int_{0}^{T}|\Delta g(s,Y_s)|dG_s\right)^{p}\right]\\
&&+\left(\E\sup_{0\leq t \leq T}|\Delta S_t|^{p}\right)^{\frac{p}{p-1}}(\Psi_{T})^{1/p},
\end{eqnarray*}
which ends the proof.
\end{proof}

\section{Existence and uniqueness of a solution}
\setcounter{theorem}{0} \setcounter{equation}{0}
With the help of the above a priori estimates, we can obtain an existence and
uniqueness result by the use of $L^{\infty}$-approximation.

Firstly, let us give this result which is a slighly extension of Theorem 3.1 of Ren and Xia \cite{Ral}.
\begin{theorem}\label{T1}
Assume $({\bf A1)}$-$({\bf A4)}$. Then RGBSDE with data
 $\left(\xi,f,g,S \right)$  has a unique solution $(Y,Z,K)
  \in \mathcal{S}^{2}
\times\mathcal{M}^{2}\times\mathcal{S}^{2}$.
\end{theorem}
To prove this theorem, we need an important result which gives an approximation of continuous functions
by Lipschitz functions (see Lepeltier and San Martin \cite{LS} to appear for the proof).
\begin{lemma}\label{Lemma3.1}
Let $f : \R^p\rightarrow\R$ be a continuous function with linear growth, that is, there exists a constant
$K < \infty$ such that $\forall \,x\in \R^p, |f(x)|\leq C(1+|x|)$. Then the sequence of functions $f_n(x) = \inf_{y\in \Q^p}\{f(y)+n|x-y|\}$
is well defined for $n \geq K$ and satisfies
\begin{description}
\item $(a)$\;  Linear growth: $\forall \, x\in\R^p,\;\; |f_n(x)|\leq M(1 + |x|)$,
\item $(b)$ \;  Monotonicity: $\forall \, x\in\R^p,\;\; f_n(x)\nearrow$,
\item $(c)$\;  Lipschitz condition: $\forall \, x, y\in\R^p,\;\; |f_n(x) - f_n(y)|\leq n|x - y|$,
\item $(d)$\,  Strong convergence: if $x_n\rightarrow x$\; as \;$n\rightarrow\infty$, then $f_n(x_n)\rightarrow f(x)$ as $n\rightarrow\infty$.
\end{description}
\end{lemma}
{\it Proof of Theorem $\ref{T1}$}
Consider, for fixed $(t,\omega)$, the sequence $(f_n(t,\omega,y,z), g_n(t,\omega,y))$ associated to $(f,g)$ by Lemma $\ref{Lemma3.1}$. Then, $f_n, g_n$ are measurable functions as well as Lipschitz functions. Moreover, since $\xi$ satisfy $({\bf A4})$ and $\{S_t, 0\leq t\leq T\}$
satisfy $({\bf A5})$, we get from Ren and Xia \cite{Ral} that there is a unique triple $\{(Y^n_t,Z^n_t,K^n_t) , 0\leq t\leq T\}$ of ${\cal F}_t$-progressively measurable processes taking values in $\R\times\R^{d}\times\R_+$ and satisfying
\begin{description}
\item $\left(i\right)$  $Y^n$  is a continuous process, \item
$\left(ii\right)$\ $Y_{t}^n=\xi
+\int_{t}^{T}f_n(s,Y^n_{s},Z_{s}^n)ds+\int_{t}^{T}g_n(s,Y_{s}^n)dG_{s}
-\int_{t}^{T}Z^n_{s}dW_{s}+K^n_{T}-K^n_{t},$ \item $\left(iii\right)$\
$Y_{t}^n\geq S_{t}$ \;\; a.s., \item $\left(iv\right)$ $\E\left(\sup_{0\leq
t\leq T}|Y_{t}^n|^{p}+\int_{0}^{T}|Z_{s}^n|^{2}ds\right)<+\infty$, \item$\left(v\right)$\ $K^n$ is a non-decreasing
process such that $K^n_{0}=0$ and
$\int_{0}^{T}(Y_{s}^n-S_{s}^n)dK^n_{s}=0$,\;\; a.s.
\end{description}
Using the comparison theorem of BSDE's in El Karoui et al. \cite{Kal1}, we obtain that
\begin{eqnarray}
\forall\, n\geq m\geq M,\;\, Y^n\geq Y^m,\;  dt\otimes d\P
\mbox{-a.s}.\label{thoecomparison}
\end{eqnarray}
The idea of the proof of Theorem $\ref{T1}$  is to establish that the limit of the sequence $(Y^n,Z^n, K^n)$ is a solution of
the RGBSDE $(\ref{aa})$ with parameters $(\xi,f,g,S)$. It follows by the same step and technics as in \cite{Matoussi}, hence we will outline.

First, there exists a constant $C$ depending only on $M,\, T, \E(\xi^2)$ and $\E(\sup_{0\leq t\leq T}(S_t^{+})^2)$, such that
\begin{eqnarray}
\E\left(\sup_{0\leq t\leq T}|Y^{n}_{t}|^2+\int_{0}^{T}|Z^{n}_s|^2 ds\right)\leq C.\label{estRGBSDE}
\end{eqnarray}

Now, we have from $(\ref{thoecomparison})$ and $(\ref{estRGBSDE})$ respectively, the existence of the process $Y$ such that $Y^n_t\nearrow Y_t,\; 0\leq t\leq T,\; \P$-a.s. and from Fatou's lemma, together with the dominated convergence theorem provide respectively
\begin{eqnarray}
\E\left(\sup_{0\leq t\leq T}|Y^{n}_{t}|^2\right)\leq C\;\; \mbox{and}\;\;\int_{0}^{T}|Y^{n}_s-Y_s|^2( ds+dG_s)\rightarrow 0\label{dominate}
\end{eqnarray}
as $n\rightarrow\infty$.

Now, we should prove that the sequence of processes $Z^n$ converge in $\cal{M}^2$. For all $n\geq m\geq n_{0}\geq M$, from
lt\^{o}'s formula for $t=0$
\begin{eqnarray*}
\E|Y^{n}_0-Y^m_0|^2+\E\int^{T}_0|Z^n_s-Z^m_s|^2ds&=&2\E\int^{T}_0(Y^n_s-Y^m)(f_n(s,Y^n_s,Z^n_s)-f_m(s,Y^m_s,Z^m_s))ds\\
&&+2\E\int^{T}_0(Y^n_s-Y^m)(g_n(s,Y^n_s)-g_m(s,Y^m_s))dG_s\\
&&+2\E\int^{T}_0(Y^n_s-Y^m)(dK^n_s-dK^m_s).
\end{eqnarray*}
Using the fact that for all $n,\ Y^n_t\geq S_t,\; 0 \leq t \leq T$, and from the identity $\int^{T}_{0}(Y^n_t - S_t)dK^n_t = 0$, we have
\begin{eqnarray*}
\E\int^{T}_0|Z^n_s-Z^m_s|^2ds&\leq&2\left(\E\int_{0}^{T}|Y^n_s-Y^m_s|^2ds\right)^{1/2}\E\left(\int_{0}^{T}|f_n(s,Y^n_s,Z^n_s)-f_m(s,Y^m_s,Z^m_s)|^2ds\right)^{1/2}\\
&&+2\left(\E\int_{0}^{T}|Y^n_s-Y^m_s|^2dG_s\right)^{1/2}\E\left(\int_{0}^{T}|g_n(s,Y^n_s)-g_m(s,Y^m_s)|^2dG_s\right)^{1/2},
\end{eqnarray*}
where we have used the H\"{o}lder inequality. By the uniform linear growth condition on the sequence $(f_n,g_n)$ and in virtue of $(\ref{estRGBSDE})$, we obtain the existence of a constant $C$ such that
\begin{eqnarray*}
\forall\, n, m\geq n_0,\;  \E\int^{T}_0|Z^n_s-Z^m_s|^2ds\leq C\E\left(\int_{0}^{T}|Y^{n}_s-Y_s^m|^2 (ds+dG_s)\right).
\end{eqnarray*}
Then from $(\ref{dominate})$, $(Z^n)$ is a Cauchy sequence in $\cal{M}$, and there exists a $\cal{F}_t$-progressively measurable process
$Z$ such that $Z^n\rightarrow Z$ in $\cal{M}^2$, as $n\rightarrow\infty$.

Similarly by It\^{o}'s formula and Davis-Burkholder-Gundy inequality, it follows that
\begin{eqnarray*}
\E\left(\sup_{0\leq t\leq T}|Y^{n}_s-Y_s^m|^2\right)\rightarrow 0
\end{eqnarray*}
as $n,m\rightarrow\infty$, from which we deduce that $\P$-almost surely, $Y^n$ converges uniformly in $t$ to $Y$ and that $Y$ is a
continuous process.

Now according to RGBSDE $(ii)$, and use the same argument as \cite{Matoussi}, we have for all $n, m\geq n_0\geq M$,
we have
\begin{eqnarray*}
\E\left(\sup_{0\leq t\leq T}|K^{n}_s-K_s^m|^2\right)\rightarrow 0
\end{eqnarray*}
as $n,m\rightarrow\infty$. Consequently, there exists a progressively measurable, increasing (with $K_0 = 0$) and a continuous process process $K$ with value in $\R_+$ such
\begin{eqnarray*}
\E\left(\sup_{0\leq t\leq T}|K^{n}_s-K_s|^2\right)\rightarrow 0
\end{eqnarray*}
as $n\rightarrow\infty$.

Finally, taking limits in the RGBSDE $(ii)$ we obtain that the triple $\{(Y_t, Z_t, K_t), \; 0\leq t\leq T\}$ is a solution of
the RGBSDE $(\ref{RGBDSEinitial})$ and satisfy
\begin{description}
\item $\left(1\right)$\
$Y_{t}\geq S_{t}$ \;\; a.s., \item $\left(2\right)$ $\E\left(\sup_{0\leq
t\leq T}|Y_{t}|^{2}+\int_{0}^{T}|Z_{s}|^{2}ds\right)<+\infty$, \item$\left(3\right)$
$\int_{0}^{T}(Y_{s}-S_{s})dK_{s}=0$,\;\; a.s.
\end{description}
$\square$

We now prove our existence and uniqueness result.
\begin{theorem}
Assume $({\bf A1)}$-$({\bf A4)}$. Then RGBSDE with data
 $\left(\xi,f,g,S \right)$  has a unique solution $(Y,Z,K)\in \mathcal{S}^{p}\times\mathcal{M}^{p}\times\mathcal{S}^{p}$.
\end{theorem}
\begin{proof}
\noindent{\bf Uniqueness}
\newline
Let us consider $(Y,Z,K)$ and $(Y',Z',K')$ two solutions of RGBSDE with data $\left(\xi,f,g,S \right)$ in the appropriate space. Using Lemma 2.4 (since $\Delta S= 0 \in L^{p},\, \Delta\xi=\Delta f=\Delta g=0$), we obtain immediately $Y = Y'$.
Therefore we have also $Z = Z'$ and finally $K = K'$, whence uniqueness follows.

Let us turn to the existence part. In order to simplify the calculations, we will
always assume that condition $({\bf A2}$-$iv)$ is satisfied with $\mu\leq 0$. If it is not true, the change
of variables $\tilde{Y}_t =e^{\mu\ t}Y_t,\; \tilde{Z}_t= e^{\mu\ t}Z_t,\; \tilde{K}_t= e^{\mu\ t}K_t$ reduces to this case

{\bf Existence} Since, the function $f$ is non-Lipschitz, the proof will be split into two steps
\newline
{\bf Step 1.} In this part $\xi,\, \sup f^{0}_{t},\, \sup g^{0}_{t},\,
\sup S_{t}^{+}$ are supposed bounded random variables and $r$ a
positive real such that
\begin{eqnarray*}
\sqrt{e^{(1+\lambda^{2})T}}(\|\xi\|_{\infty}+T\|f^{0}
\|_{\infty}+\|G_{T}\|_{\infty}\|g^{0}\|_{\infty}
+\|S^{+}\|_{\infty}) &<& r.
\end{eqnarray*}
Let $\theta_{r}$ be a smooth function such that $0\leq
\theta_{r}\leq 1$ and
\begin{eqnarray*}
\theta_{r}(y)=\left\{ \begin{array}{l}
1\text{ for }|y|\leq r \\
\\
0\text{ for }|y|\geq r+1.
\end{array}
\right.
\end{eqnarray*}
For each $n\in\N^{*}$, we denote $q_{n}(z)=z\frac{n}{|z|\vee n }$ and set
\begin{eqnarray*}
h_{n}(t,y,z)&=&\theta_{r}(y)(f(t,y,q_{n}(z))-f_{t}^{0})\frac{n}{\pi_{r+1}
(t)\vee n}+f_{t}^{0}.
\end{eqnarray*}
According to the same reason as in \cite{BDPS}, this function still satisfies quadratic condition $({\bf A2}$-$iv)$ but with a positive constant i.e
there exists $\kappa >0 $ depending on $n$ such that
\begin{eqnarray*}
(y-y')(h_{n}(t,y,z)-h_{n}(t,y',z))&\leq& \kappa|y-y'|^2.
\end{eqnarray*}
Then $(\xi,h_{n},g,S)$ satisfies assumptions of Theorem 3.1. Hence, for each
$n\in\N$, the reflected generalized BSDE associated to $(\xi,h_{n},g,S)$ has a unique solution
$(Y^{n},Z^{n},K^{n})$ belong in space $\mathcal{S}^{2}
\times\mathcal{M}^{2}\times\mathcal{S}^{2}$.\newline Since
\begin{eqnarray*}
y\ h_{n}(t,y,z)&\leq& |y|\ \|f^{0}\|_{\infty}+\lambda|y|\ |z|
\end{eqnarray*}
and $\xi,\ S$ and $G$ are bounded, the similar computation of Lemma $2.2$ in
\cite{BC} provide that the process $Y^{n}$ satisfies the inequality $\|Y^{n}\|_{\infty}\leq r$. In addition, from Lemma 2.2, $\|Z^{n}\|_{\mathcal{M}^{2}}\leq
r'$ where $r'$ is another constant. As a byproduct
$(Y^{n},Z^{n},K^{n})$ is a solution to the reflected generalized
BSDE associated to $(\xi,f_{n},g,S)$ where
\begin{eqnarray*}
f_{n}(t,y,z)&=&(f(t,y,q_{n}(z))-f_{t}^{0})\frac{n}{\pi_{r+1}
(t)\vee n}+f_{t}^{0}
\end{eqnarray*}
which satisfied assumption $({\bf A2}$-$iv)$ with $\mu\leq 0$.
\newline
We now have, for $i\in\N$, setting
$
\begin{array}{l}
\bar{Y}^{n,i}=Y^{n+i}-Y^{n},\ \bar{Z}^{n,i}=Z^{n+i}-Z^{n},\
\bar{K}^{n,i}=K^{n+i}-K^{n},
\end{array}
$
applying the similar argument as Lemme 2.3, we obtain
\begin{eqnarray*}
&&\Phi(t)|\bar{Y}_{t}^{n,i}|^{2}+\frac{1}{2}\int_{t}^{T}
\Phi(s)|\bar{Z}_{s}^{n,i}|^{2}ds  \\
&\leq&
2\int_{t}^{T}\Phi(s)\bar{Y}_{s}^{n,i}(f_{n+i}(s,Y_{s}^{n},Z_{s}^{n})
-f_{n}(s,Y_{s}^{n},Z_{s}^{n}))ds \\
&&+2\int_{t}^{T}\Phi(s)\bar{Y}_{s}^{n,i}d\bar{K}_{s}^{n}-
2\int_{t}^{T}\Phi(s)\bar{Y}_{s}^{n,i}\bar{Z}^{n,i}dW_{s},
\end{eqnarray*}
where for $\alpha>0,\,\Phi(s)=\exp(2\lambda^{2} s)$.
But $\|\bar{Y}^{n,i}\|_{\infty}\leq 2r$ so that
\begin{eqnarray*}
&&\Phi(t)|\bar{Y}_{t}^{n,i}|^{2}+\frac{1}{2}\int_{t}^{T}
\Phi(s)|\bar{Z}_{s}^{n,i}|^{2}ds  \\
&\leq& 4r\int_{t}^{T}\Phi(s)|f_{n+i}(s,Y_{s}^{n},Z_{s}^{n})
-f_{n}(s,Y_{s}^{n},Z_{s}^{n})|ds\\
&&+2\int_{t}^{T}\Phi(s)\bar{Y}_{s}^{n,i}d\bar{K}_{s}^{n,i}-
2\int_{t}^{T}\Phi(s)\bar{Y}_{s}^{n,i}\bar{Z}^{n,i}dW_{s}
\end{eqnarray*}
and using the BDG inequality, we get, for a constant $C$ depending only on $\lambda,\ \mu$ and $T$,
\begin{eqnarray}
&&\E\left(\sup_{0\leq t\leq T}
|\bar{Y}_{t}^{n,i}|^{2}+\int_{0}^{T}
|\bar{Z}_{s}^{n,i}|^{2}ds\right) \nonumber \\&\leq&
Cr\E\left\{\int_{0}^{T}|f_{n+i}(s,Y_{s}^{n},Z_{s}^{n})
-f_{n}(s,Y_{s}^{n},Z_{s}^{n})|ds\right\}.\label{c4}
\end{eqnarray}
On the other hand, since $\|Y^{n}\|_{\infty}\leq r$, we get
\begin{eqnarray*}
|f_{n+i}(s,Y_{s}^{n},Z_{s}^{n})-f_{n}(s,Y_{s}^{n},Z_{s}^{n})|
&\leq& 2\lambda|Z^{n}_s|{\bf 1}_{\{|Z^{n}_s|\ >n\}}
+2\lambda|Z^{n}_s|{\bf
1}_{\{\pi_{r+1}(s)>n\}}\\
&&+2\pi_{r+1}(s){\bf 1}_{\{\pi_{r+1}(s)>n\}}
\end{eqnarray*}
from which we deduce, according assumption $(\bf A3)$ and
inequality $(\ref{c4})$ that $(Y^{n},Z^{n})$ is a cauchy sequence
in the Banach space $\mathcal{S}^{2}\times \mathcal{M}^{2}$. Let $(Y,Z)$ its limit in $\mathcal{S}^{2}\times \mathcal{M}^{2}$,
then for all $0\leq t\leq T$,\,$Y_{t}\geq S_{t}$ a.s..

Next, let us define
\begin{eqnarray}
K_{t}^{n}=Y_{0}^n-Y_{t}^n-\int_{0}^{t}f_n(s,Y_{s}^{n},Z_{s}^{n})ds
-\int_{0}^{t}g(s,Y_{s}^{n})dG_{s}+\int_{0}^{t}Z_{s}^{n}dW_{s}.
\label{K}
\end{eqnarray}
By the convergence of $Y^n,$ (for a subsequence), the fact that
$f, g$ are continuous and
\begin{itemize}
\item $\sup_{n\geq0}|f(s,Y^n_s,Z_s)|\leq f_s+K\left\{(\sup_{n\geq0}|Y_s^n|)+|Z_s|\right\}$,
\item $\sup_{n\geq0}|g(s,Y^n_s)|\leq g_s+K\left\{(\sup_{n\geq0}|Y^n_s|)\right\}$
\item $\mathbb{E}\int_0^T|f(s,Y^n_s,q_n(Z^n_s))-f(s,Y^n_s,Z_s)|^2ds\leq C\mathbb{E}\int_0^T|q_n(Z^n_s)-Z_s|^2ds$
\end{itemize}
we get the existence of a  process $K$ which verifies for all
$t\in[0,T]$
$$\E\left| K_{t}^{n}-K_{t}^{{}}\right| ^{2}\longrightarrow 0.$$

Moreover
\begin{eqnarray*}
\int_{0}^{T}(Y_{s}-S_{s})dK_{s}=0,\ \mbox{for every}\  T\geq
0.
\end{eqnarray*}
It is easy to pass to the limit in the
approximating equation associated to $(\xi,f_{n},g,S)$, yielding $(Y,Z,K)$ as a solution of reflected generalized BSDE associated to data
$(\xi,f,g,S)$.

\noindent{\bf Step 2.}\ We now treat the general
case.

For each $n\in\N^{*}$, let us denote
\begin{eqnarray*}
\xi_{n}= q_{n}(\xi),\,f_{n}(t,y,z)= f\left(t,y,z\right)-f^{0}_{t}+q_{n}(f^{0}_{t}),\\
g_{n}(t,y)= g\left(t,y\right)-g^{0}_{t}+q_{n}(g^{0}_{t}),\,S_{t}^{n}= q_{n}(S_{t}).
\end{eqnarray*}
For each $n\in\N^{*}$, RGBSDE associated with $(\xi_{n},f_{n},g_{n},S^{n})$ has a
unique solution $(Y^{n},Z^{n},K^{n})\in L^{2}$\ thanks to the
first step of this proof, but in fact also in  $L^{p},\ p\ > \ 1$ according the Lemma 2.3. Now from Lemma 2.4, for $(i,n)\in \N\times\N^{*}$,
\begin{eqnarray*}
&&\E\left\{\sup_{0\leq t\leq T}|Y_{t}^{n+i}-Y_{t}^{n}|^{p}+
\left(\int_{0}^{T}|Z_{s}^{n+i}-Z_{s}^{n}|^{2}ds\right)^{p/2}\right\}\\
&\leq& C\E\left\{|\xi_{n+i}-\xi_{n}|^{p}+\int_{0}^{T}
|q_{n+i}(f^{0}_{s})-q_{n}(f^{0}_{s})|^{p}ds\right.\\
&&+\left.\int_{0}^{T}
|q_{n+i}(g^{0}_{s})-q_{n}(g^{0}_{s})|^{p}dG_{s}+\sup_{0\leq t\leq
T}|q_{n+i}(S_{t})-q_{n}(S_{t})|^{p}\right\},
\end{eqnarray*}
where $C$ depends  on $T$ and $\lambda $.
The right-hand side of the last inequality clearly tends to $0$ as $n \longrightarrow \infty,$
uniformly on $i$  so that $(Y^{n},Z^{n})$ is again a cauchy sequence in
$\mathcal{S}^{p}\times\mathcal{M}^{p}$. Let us denote by $(Y,Z)\in \mathcal{S}^{p}\times\mathcal{M}^{p}$ it limit.
Then it follows from identical computation as previous that,
there exists a non-decreasing process $K (K_{0}=0)$ such that
\begin{eqnarray*}
\E\left(
|K_{t}^{n}-K_{t}|^{p}\right)\longrightarrow 0, as \ n
\longrightarrow \infty
\end{eqnarray*}
and
\begin{eqnarray*}
\int_{0}^{T}(Y_{s}-S_{s})dK_{s}=0,\ \mbox{for every}\  T\geq
0.
\end{eqnarray*}
It is easy to pass to the limit in the
approximating equation, yielding that the triplet $(Y,Z,K)$ is a $L^{p}$-solution of RGBSDEs with determinist time associated to $(\xi,f,g,S)$.
\end{proof}

\label{lastpage-01}
%

%
\bibliography{database}                               %
\end{document}